\documentclass[11pt]{article}
\usepackage{amssymb,amsmath,epsfig,color,array,graphics,graphicx,amsthm}

\newtheorem{lemma}{Lemma}[section]

\newtheorem{theorem}[lemma]{Theorem}

\newtheorem{definition}[lemma]{Definition}
\usepackage[utf8]{inputenc}

\DeclareMathOperator{\optw}{{\bf tw}}
\newcommand{\tw}[1]{{\optw({#1})}}
\newcommand{\NP}{\ensuremath{\mathrm{\sf NP}}}

\newcommand{\FPT}{\ensuremath{\mathrm{\sf FPT}}}
\newcommand{\bigO}[1]{\mathcal O(#1)}
\newcommand{\defterm}[1]{\emph{#1}}

\begin{document}

\title{Contracting planar graphs to contractions of triangulations%
\thanks{This research was done while the last two authors were visiting the  Dèpartement d'Informatique of Université Libre de Bruxelles in January 2010.
The authors thank Samuel Fiorini for his kind support. The work is also supported by the Actions de 
Recherche Concertées (ARC) fund of the Communauté française de Belgique,
by the project ``Kapodistrias'' (A${\rm \Pi}$ 02839/28.07.2008) of the National and
Kapodistrian University of Athens and by EPSRC Grant EP/G043434/1.
}}

\author{Marcin Kami\'nski\thanks{D\'epartement d'Informatique,
Universit\'e Libre de Bruxelles,
Boulevard du Triomphe CP212, B-1050 Brussels, Belgium.
Email: {\tt marcin.kaminski@ulb.ac.be}}
\and Dani\"el Paulusma\thanks{School of Engineering and Computing Sciences, 
Science Laboratories, South Road, Durham DH1 3LE, England
Email: {\tt daniel.paulusma@durham.ac.uk}}
\and Dimitrios M. Thilikos\thanks{Department of Mathematics,
National and Kapodistrian University of Athens,
Panepistimioupolis, GR15784 Athens, Greece. Email: {\tt sedthilk@math.uoa.gr}}
}             

\date{}
\maketitle

\begin{abstract}
For every graph $H$, there exists a poly\-no\-mial-time algorithm deciding if a planar input graph $G$ can be contracted to~$H$. However, the degree of the polynomial depends on the size of $H$. 
In this paper, we identify a class of graphs $\cal C$ such that for every $H \in \cal C$, there exists an algorithm deciding in time $f(|V(H)|) \cdot |V(G)|^{\bigO{1}}$ whether a planar graph $G$ can be contracted to~$H$. (The function $f(\cdot)$ does not depend on $G$.)  
The class $\cal C$ is the closure of planar triangulated graphs under taking of contractions. In fact, we prove that a graph $H \in \cal C$ if and only if there exists a constant $c_H$ such that if the tree-width of a graph is at least $c_H$, it contains $H$ as a contraction. We also provide a characterization of $\cal C$ in terms of minimal forbidden contractions.

\end{abstract}

\section{Introduction}

We consider simple graphs without loops and multiple edges. For a graph $G$, let $V(G)$ be its vertex set and $E(G)$ its edge set. For notions not defined here, we refer the reader to the monograph \cite{Diestel}.

\subsection{Planar graphs} 

All graphs in this paper are planar. Plane graphs are always assumed to be drawn on the unit sphere and their edges are arbitrary polygonal arcs (not necessarily straight line segments). 

\medskip
\noindent{\it Embeddings.} In this work, we only need to distinguish between essentially different embeddings of a planar graph. This motivates the following definition. 

Two plane graphs $G$ and $H$ are \defterm{combinatorially equivalent} ($G \simeq H$) if there exists a homeomorphism of the unit sphere (in which they are embedded) which transforms one into the other. The relation of being combinatorially equivalent is reflexive, symmetric and transitive, and thus an equivalence relation. Let $\mathcal G$ be the class of all plane graphs isomorphic to a planar graph $G$ and let us consider the quotient set ${\cal G} / \simeq$. The equivalence classes (i.e., the elements of the quotient set) can be thought of as \emph{embeddings}. In fact, we will work with embeddings but for simplicity, we will pick a plane graph representative for each embedding.

\medskip
\noindent{\it Dual.} The dual of a plane graph $G$ will be denoted by $G^*$. Note that there is a one-to-one correspondence between the edges of $G$ and the edges of $G^*$. We keep the convention that $e^*$ is the edge of $G^*$ corresponding to edge $e$ of $G$.

\medskip
\noindent{\it Triangulation.} A planar graphs is called \emph{triangulated} if it has an embedding in which every face is incident with exactly three vertices. Let us recall two useful facts related to planar 3-connected graphs that we will need later. 

\begin{lemma}\label{lem-triangulated-are-3-conn}
Triangulated planar graphs are 3-connected.
\end{lemma}

\begin{lemma}\label{lem-3-conn-unique}
A 3-connected planar graph has a unique embedding.
\end{lemma}

For proofs of these lemmas, see for instance \cite{MoharThomassen}: Lemma 2.3.3, p. 31 and Lemma 2.5.1, p.39, respectively. From these two lemmas, every triangulated graph has a unique embedding. 

\medskip
\noindent{\it Grids and walls.} The $k\times k$ {\it grid} $M_k$ has as its vertex set all pairs $(i,j)$ for $i,j = 0,1,\ldots, k-1$, and two vertices $(i,j)$ and $(i^\prime,j^\prime)$ are joined by an edge if and only if $|i-i^\prime| + |j-j^\prime| = 1$. 

For $k\geq 2$, let $\Gamma_k$ denote the graph obtained from $M_k$ by triangulating its faces as follows: add an edge between vertices $(i,j)$ and $(i',j')$ if $i-i'=1$ and $j'-j=1$, and
add an edge between corner vertex $(k-1,k-1)$ and every external vertex that is not already adjacent to $(k-1,k-1)$, i.e., every vertex $(i,j)$ with $i\in \{0,k-1\}$ or $j\in \{0,k-1\}$, apart from the vertices $(k-2,k-1)$ and $(k-1,k-2)$. The graph $\Gamma_k$ a called a {\it triangulated grid}. 
See Figure~\ref{f-examples} for the graphs $M_6$ and $\Gamma_6$.
The dual  $\Gamma^*_k$ of a triangulated grid is called a \emph{wall}.

\begin{figure}
  \centering
  \includegraphics[scale=.85]{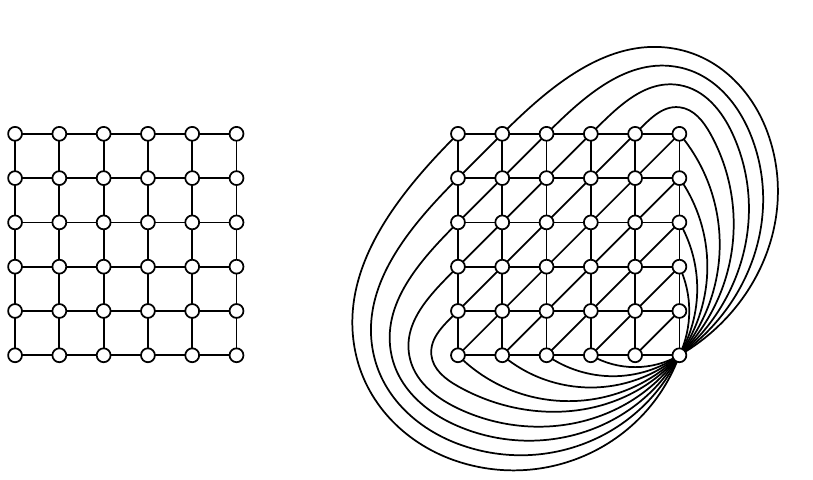}
  \caption{The graphs $M_6$ and $\Gamma_6$, respectively.}\label{f-examples}
\end{figure}

\medskip
\noindent{\it Tree-width.} We recall some results related to tree-width that we will need later in the paper. MSOL is the monadic second order logic.

\begin{lemma}[(6.2) in \cite{RobertsonST94}]\label{lem-grid} 
Let $m \geq 1$ be an integer. Every planar graph with no $m \times m$ grid minor has tree-width $\leq 6m - 5$.
\end{lemma}

\begin{lemma}[(1.5) in \cite{RobertsonST94}]\label{lem-planar-in-grid} 
If $H$ is a planar graph with $|V(H)| + 2 |E(H)| \leq n$, then $H$ is isomorphic to a minor of the $2n \times 2n$ grid.
\end{lemma}

\begin{lemma}[Theorem 6 in \cite{BouchitteMT03}]\label{lem-twd-dual} 
For any plane graph $G$ and its dual $G^*$, $\tw{G^*} \leq \tw{G} + 1$.
\end{lemma}

\begin{lemma}[\cite{Bodlaender96}]\label{lem-twd-linear}
For every fixed $k$, there exists a linear time algorithm deciding whether the input graph has tree-width at most $k$.
\end{lemma}

\begin{lemma}[\cite{Courcelle90}]\label{lem-courcelle}
For every fixed $k$ and a problem $\cal P$ expressible in MSOL, there exists a linear time algorithm for $\cal P$ in the class of graphs of tree-width at most $k$.
\end{lemma}

\medskip
\noindent{\it Pasting along vertices and edges.} If $G$ is a graph with induced subgraphs $G_1$ and $G_2$ such that $G = G_1 \cup G_2$ we say that $G$ arises from $G_1$ and $G_2$ by {\it pasting along} $G_1 \cap G_2$. In this paper, pasting along vertices and edges (that is, if $G_1 \cap G_2$ is a vertex or an edge) is particularly important. We note that pasting planar graphs along vertices and edges creates planar graphs.

\subsection{Containment relations} 

An \defterm{edge contraction} of an edge $e$ in a graph is the graph obtained by removing $e$, identifying its two endpoints, and eliminating parallel edges that may appear. Some basic properties of contractions are collected in \cite{WolleB04}. Formally, for an edge $e$ with endpoints $u$ and $w$, the contraction of $e$, denoted by $G / e$, is the graph with vertex set $V(G / e) = V(G) \setminus \{u,w\} \cup \{v_{uw}\}$ and edge set  

\begin{center}
\begin{tabular}{ccl}
$E(G / e) =  E $ & $\setminus$ & $\{ \; \{x,y\} \in E \; : \; x \in \{u,w\}, \; y \in V  \; \}$  \\
 & $\cup$ & $ \{ \; \{v_{uw}, x \}\; : \;  \{x,u\}\in E \; \vee\; \{x,w\} \in E \; \}.$
\end{tabular}
\end{center}

A graph $H$ is a {\em contraction} of a graph $G$ (or $G$ is \defterm{contractible} to $H$) if $H$ can be obtained from $G$ by a sequence of edge contractions. We denote it by $H \leq_c G$. 

A \defterm{dissolution} of a vertex $v$ of degree 2 in a plane graph $G$ is a contraction of one of the two edges $v$ is incident with in $G$. A graph $H$ that can be obtained from $G$ be a sequence of dissolutions  and edge/vertex deletions is called a \defterm{topological minor} of $G$ ($H <_{tm} G$). Finally, a graph $H$ that is a contraction of a subgraph of $G$ is called its \defterm{minor} ($H <_m G$). 

In the paper, when we speak about different containment relations like contraction or topological minor, the graph $H$ will be called a \emph{pattern}.

\subsection{Parameterized complexity} 

Parameterized complexity is a paradigm in computational complexity and analysis of algorithms that has received a lot of attention in the last 20 years. The idea is to
evaluate the efficiency of an algorithms not only in terms of the size of the input $n$ but also by some parameter $k$ of it. The algorithms whose running time can be bounded by a function $f(k) \cdot n^{\bigO{1}}$ are considered {\sl efficient} from the parameterized complexity point of view and called \defterm{fixed parameter tractable} or \FPT in contrast to algorithms with time complexities $n^{f(k)}$ or worst. For more information on parameterized complexity we refer to \cite{DowneyFellows,FlumG06, Niedermeier-book}.

\section{Previous work} 

The problem of checking whether a graph is a contraction of another has already attracted some attention. In this section we briefly survey known results.
Let $P_n, C_n$ and $K_n$ denote the path, cycle and complete graph on $n$ vertices, respectively.
Let $K_{p,q}$ denote the complete bipartite graph with partition classes of 
size $p$ and $q$, respectively. The graph $K_{1,m}$ for $m\geq 1$ is also
called a {\it star}.

Perhaps the first systematic study of contractions was undertaken by Brouwer and Veldman \cite{BrouwerV87}. Here are two main theorems from that paper.

\begin{theorem}[Theorem 3 in \cite{BrouwerV87}]
A graph $G$ is contractible to $K_{1,m}$ if and only if $G$ is connected and contains an independent set $S$ of $m$ vertices such that $G-S$ is connected.
\end{theorem}

In particular, a graph is contractible to $P_3$ if and only if it is connected and is neither a cycle nor a complete graph. The theorem also allows to detect, in polynomial time, if a graph is contractible to $K_{1,m}$. It suffices to enumerate over all sets $S$ with $m$ independent vertices and check if the graph $G-S$ is connected. This gives an $|V(G)|^{\bigO{m}}$ algorithm, which is polynomial for every fixed $m$.

\begin{theorem}[Theorem 9 in \cite{BrouwerV87}]
If $H$ is a connected triangle-free graph other than a star, then contractibility to $H$ is \NP-complete.
\end{theorem}

Hence, checking if a graph is contractible to $P_4$ or $C_4$ is \NP-complete. More generally, it is \NP-complete for every bipartite graph with at least one connected component that is not a star. 

The research direction initiated by Brouwer and Veldman was continued by Levin, Paulusma, and 
Woeginger~\cite{LevinPW08,LevinPW08a}. Here is the main result established in these two papers.

\begin{theorem}[Theorem 3 in \cite{LevinPW08}]
Let $H$ be a connected graph on at most 5 vertices. If $H$ has a dominating vertex, then contractibility to $H$ can be decided in polynomial time. If $H$ does not have a dominating vertex, then contractibility to $H$ is \NP-complete.
\end{theorem}

However, the existence of a dominating vertex in the pattern $H$ is not enough to ensure that contractibility to $H$ can be decided in polynomial time. A pattern on 69 vertices for which contractibility to $H$ is \NP-complete was exhibited in \cite{HofKPST10}.

Looking at contractions to fixed pattern graphs is justified by the following theorems proved by Matou\v sek and Thomas in \cite{MatousekT92}. 

\begin{theorem}[Theorem 4.1 in \cite{MatousekT92}]
The problem of deciding, given two input graphs $G$ and $H$, whether $G$ is contractible to $H$ is \NP-complete even if we impose one of the following restrictions on $G$ and $H$:
\begin{description}
\item[(i)] $H$ and $G$ are trees of bounded diameter,
\item[(ii)] $H$ and $G$ are trees all whose vertices but one have degree at most $5$.
\end{description}
\end{theorem}

\begin{theorem}[Theorem 4.3 in \cite{MatousekT92}]
For every fixed $k$, the problem of deciding, given two input graphs $G$ and $H$, whether $G$ is contractible to $H$ is \NP-complete even if we restrict $G$ to partial $k$-trees and $H$ to $k$-connected graphs. 
\end{theorem}

The authors also proved a positive result. 

\begin{theorem}[Theorem 5.14 in \cite{MatousekT92}]
For every fixed $\Delta, k$, there exists an $\bigO{|V(H)|^{k+1} \cdot |V(G)|}$ algorithm to decide, given two input graphs $G$ and $H$, whether $G$ is contractible to $H$, when the maximum degree of $H$ is at most $\Delta$ and $G$ is a partial $k$-tree.
\end{theorem}
In a previous paper, we studied the problem of contracting a planar graph to a fixed pattern \cite{KaminskiPT10}. Here we will need some of the definitions and results of that paper. 

An \defterm{embedded contraction} of an edge $e$ of a plane graph $G$ is a plane graph $G'$ that is obtained by homeomorphically mapping the endpoints of $e$ in $G$ to a single vertex without any edge crossings and recursively removing one of the two parallel edges bounding a 2-face, if a graph has such a pair. Note that there are many embedded contractions of an edge of a plane graph $G$ but they are all combinatorially equivalent.

An \defterm{embedded dissolution} of a vertex $v$ of degree 2  in a plane graph $G$ is an embedded contraction of one of the two edges $v$ is incident with in $G$. 

Let $G$ and $H$ be two plane graphs. We say that $H$ is an embedded contraction of $G$ ($H \leq_{ec} G$), if $H$ is combinatorially equivalent to a graph that can be obtained from $G$ by a series of embedded contractions. We say  that $H$ is an \defterm{embedded topological minor} of $G$ ($H \leq_{etm} G$), if $H$ is combinatorially equivalent to a graph that can be obtained from $G$ by a series of vertex and edge deletions, and embedded dissolution of vertices of degree 2.

The main technical result of \cite{KaminskiPT10} is an equivalence between embedded contractions in a planar graph and embedded topological minors in its dual. (A multigraph is called \emph{thin} if it has no two parallel edges bounding a 2-face. Simple graphs are in particular thin.)

\begin{lemma}[Lemma 2 in \cite{KaminskiPT10}]\label{lem-ec-etm}
Let $H$ and $G$ be two thin planar graphs and $H^*$, $G^*$ their respective duals. 
\[ H \leq_{ec} G \iff H^* \leq_{etm} G^* \]
\end{lemma}

This equivalence is used to reduce the problem of finding a contraction in a planar graph 
$G$ to finding an embedded topological minor in its dual graph $G^*$. This consequently leads to the main result of that paper.

\begin{theorem}[Theorem 12 in \cite{KaminskiPT10}]\label{thm-main-poly}
For every graph $H$, there exists a polynomial-time algorithm that given a planar graph $G$ decides whether $H$ is a contraction of $G$, and if so finds a series of contractions transforming $G$ into $H$. 
\end{theorem}

\section{Our motivation and results}

Theorem \ref{thm-main-poly} tells us that for every graph $H$, one can decide in polynomial time if a planar graph can be contracted to $H$. However, the running time of the algorithm, as mentioned in  \cite{KaminskiPT10}, is bounded by a polynomial whose degree depends on the size of $H$. Therefore, this is not an \FPT~algorithm, when parameterized by $|H|$. 



Recently, Grohe et al.~\cite{GKMW} announced that topological minor testing is \FPT.
We emphasize that due to the difference between embedded topological minors and topological minors their
result does not imply that contraction testing is \FPT\ in planar graphs; the latter problem is still open.

In this paper, given the duality of Theorem \ref{thm-main-poly}, we focus on contractions in planar graphs and identify a class of patterns for which an \FPT~algorithm exists. 
We show that the problem of testing whether a planar graph $G$ can be contracted
to a triangulated planar graph $H$ is equivalent to testing whether the dual of $G$ contains the dual of $H$ as a topological minor; so for such cases embedded topological minors and topological minors coincide.
This means that we could use the result of Grohe et al. to immediately find that 
contraction testing is \FPT\ in planar graphs as long as the pattern graphs are triangulated.
However, we aim for a stronger result.
Let $\cal C$ be the closure of the class of triangulated planar graphs with respect to taking of contractions.
We present an FPT algorithm that tests if a planar graph can be contracted to a pattern
graph $H\in {\cal C}$.

Our approach is as follows.
We prove that for every graph $H \in {\cal C}$, there is a constant $c_H$ such that if $\tw{G} > c_H$, then $G$ contains $H$ as a contraction. 
Our FPT algorithm first checks if the tree-width of the input graph is large enough. If so, then the input graph contains $H$ as a contraction. Otherwise, we use the celebrated result by Courcelle~\cite{Courcelle90} to solve the problem on the class of graphs with bounded tree-width. 

We also study properties of $\cal C$ and provide a characterization of the class in terms of forbidden contractions.

\section{Algorithm}

\begin{definition}
Let $\cal T$ be the class of triangulated planar graphs and $\cal C$ the minimal contraction-closed class containing $\cal T$.
\end{definition}

Note that $\cal C$ is well-defined and unique. Indeed, suppose that there are two
different inclusion-minimal contraction-closed classes containing $\cal T$. Then, there exists a graph $G$ that belongs to one but not the other. However, $G$ is a contraction of some planar triangulated graph $G_T \in {\cal T}$. By definition, $G_T$ and all its contractions should belong to both classes; a contradiction. 

\begin{lemma}\label{l-pasting}
$\cal C$ is the closure of ${\cal T} \cup \{K_2\}$ with respect to pasting along vertices and edges.
\end{lemma}

\begin{proof}
For the forward implication, let us note that triangulated planar graphs are 3-connected by Lemma \ref{lem-triangulated-are-3-conn}. All cut-sets of a triangulated planar graph are isomorphic to the clique on 3 vertices. Let us observe that contractions of a graph whose cut-sets are cliques are graphs with clique cut-sets. Also, the size of the cut-set will not increase after contraction. 

Let $G \in {\cal C}$. By definition of $\cal C$, $G$ is a contraction of a triangulated planar graph. The maximal 3-connected components of $G$ are triangulated and the cut-sets of size less than 3 are isomorphic to complete graphs. Hence, $G$ can be obtained from triangulated graphs by pasting along vertices and edges.

For the backward implication, let $G$ be a minimal graph that belongs to the closure of $\cal T$ with respect to pasting along vertices and edges but is not a contraction of a triangulated planar graph. 

First, suppose that $G$ has a cut-vertex $v$ and the connected components of $G \setminus v$ are $C_1,\ldots, C_k$. Let $C'_1,\ldots, C'_k$ be triangulated planar graphs such $C_1,\ldots, C_k$ are their contractions respectively. Let us consider drawings of $C'_1,\ldots, C'_k$ such that $v$ belongs to the outerface in each of these drawings (clearly, such drawings exist). Let $x_i, y_i$ be the other two vertices of $C_i$, for $i=1,\ldots,k$, besides $v$, incident with the outerface. Now let us identify $v$ from different drawings of $C'_1,\ldots, C'_k$ in such a way that the cyclic order of vertices around $v$ is $x_1,y_1,\ldots, x_k,y_k$. Let us then add new vertices $w_1,\ldots,w_k$ to the drawing and make them all adjacent to $v$ and each $w_i$ adjacent to $y_i$ and $x_{i+1}$, for $k=1,\ldots,k-1$, and $w_k$ adjacent to $y_k$ and $x_1$. At the last step, we add a vertex $u$ adjacent to all $x_i,y_i,w_i$, for $i=1,\ldots,k$. It is easy to verify that  the graph we have created is planar and triangulated, and that $G$ is a contraction of that graph; a contradiction.

Second, we suppose that $G$ has a cut-edge $vw$. Let us consider a planar drawing of $G$ and put a new vertex in every 4-face of $G$. Next, we make the new vertices adjacent to all 4 vertices incident with that face. It is easy to verify that  the graph we have created is planar and triangulated, and that $G$ is a contraction of that graph; a contradiction.
\end{proof}

\begin{lemma}\label{lem-c-tm-for-3-conn}
Let $H$ be a triangulated planar graph and $H^*$ be its dual. For every planar graph $G$ and its dual $G^*$,
\[ H \leq_{c} G \iff H^* \leq_{tm} G^* \]
\end{lemma}

\begin{proof}
$H$ is a triangulated planar graph and has a unique embedding by Lemmas \ref{lem-triangulated-are-3-conn} and \ref{lem-3-conn-unique}. Hence, the contraction and embedded contraction relations coincide. As the dual of a planar triangulated graph also has a unique embedding, the embedded topological minor and  topological relations minor coincide. Thus, $H \leq_{ec} G \iff H \leq_{c} G$ and $H^* \leq_{etm} G^* \iff H^* \leq_{tm} G^*$ and the lemma follows from Lemma \ref{lem-ec-etm}.
\end{proof}

\begin{theorem}\label{thm-big-twd-triang}
A graph $H \in {\cal C}$ if and only if there exists a constant $c_H$ such that for every graph $G$, if $\tw{G} > c_H$, then $H \leq_c G$.
\end{theorem}

\begin{proof}
Let $H\in {\cal C}$.
For the forward implication we note that $H$ is by definition a contraction of a triangulated planar graph. Hence, there exists a triangulated planar graph $H_T$ such that $H \leq_c H_T$. For a graph $G$, we have that $H_T \leq_{c} G \iff H_T^* \leq_{tm} G^*$ by Lemma \ref{lem-c-tm-for-3-conn}. 

$H^*_T$ is cubic and as the dual of a triangulation also has a unique embedding into the sphere. The minor and topological minor relations are equivalent for graphs of maximum degree 3 (see Proposition 1.7.2, p.20, \cite{Diestel}), therefore $H_T \leq_{c} G \iff H_T^* \leq_{m} G^*$. 

However, every planar graph is a minor of some grid by Lemma~\ref{lem-planar-in-grid}. (As explained in \cite{Diestel}, ``To see this take a drawing of the graph, fatten its vertices and edges, and superimpose a sufficiently fine plane grid.'') Let $m_{H_T^*}$ be the minimum size of a grid that contains $H_T^* $ as a minor and $c_H > 6 m_{H_T^*} - 4$. Then, for every graph $G$ of tree-width at least $c_H$, the tree-width of its dual is at least $c_H - 1 > 6 m_{H_T^*} - 5$ by Lemmas \ref{lem-grid} and \ref{lem-twd-dual}. Therefore, $G^*$ contains $H_T^*$ as a minor and, as explained before, $G$ contains $H$ as a contraction. 

For the backward implication, we need to prove that if there exists a constant $c_H$ such that every planar graph G with $\tw{G} > c_H$ contains $H$ as a contraction, for some graph $H$, then $H \in {\cal C}$. To see this, choose $G$ to be a triangulated grid with $\tw{G} > c_H$. Then, $G$ contains $H$ as a contraction, so $H \in {\cal C}$.
\end{proof}

\begin{theorem}\label{thm-FPT}
For every graph $H \in {\cal C}$, there exists an algorithm deciding in time  $f(|V(H)|) \cdot |V(G)|^{\bigO{1}}$ whether a planar graph $G$ can be contracted to $H$.
\end{theorem}

\begin{proof}
We can assume that $H$ is planar because planar graphs are closed under taking of contractions. (If $H$ is not planar, it is not a contraction of a planar graph.) Let $c_H$ be the constant from Theorem \ref{thm-big-twd-triang}. First, we test if $\tw{G} > c_H$ using the linear-time algorithm from Lemma \ref{lem-twd-linear}. If so, then by Theorem \ref{thm-big-twd-triang}, $G$ contains $H$ as a contraction. Otherwise, we use Lemma \ref{lem-courcelle}. We point out that testing whether a graph can be contracted to a fixed pattern is expressible MSOL using standard techniques (connectivity and adjacency are needed and both are expressible in MSOL).
\end{proof}

\section{Characterizations of $\cal C$}

In this section, we give two equivalent characterizations of $\cal C$. First, we show that a graph belongs to $\cal C$ if and only if it is a contraction of a triangulated grid. Then, we provide a characterization of $\cal C$ in terms of minimal forbidden contractions.  

\subsection{Triangulated grids}

\begin{lemma}\label{lem-traingulated-grid}
$G \in {\cal C}$ if and only if $G$ is a contraction of a triangulated grid. 
\end{lemma}

\begin{proof}
For the forward implication, let us recall that by Lemma~\ref{lem-planar-in-grid}, for every planar graph $H$, there exists a constant $m_H$ such that $H$ is a minor of the $m_H \times m_H$ grid. In fact, one can easily see that a grid is a minor of some wall. Hence, for every planar graph $H$, there exists a constant $m'_H$ such that $H$ is a minor of the $\Gamma^*_{m'_H}$ wall. In particular, every cubic planar graphs is a minor of a big enough wall. Since the minor and topological minor relations are equivalent for graphs of maximum degree 3 (see Proposition 1.7.2, p.20, \cite{Diestel}), every cubic planar graphs is a topological minor of a big enough wall.

If $G \in {\cal C}$, then there exists a triangulated planar graph $H_T$ such that $H \leq_c H_T$. Let $G^*$ be a wall big enough that it contains $H_T^*$ as a topological minor ($H_T^*$ is cubic). From Lemma \ref{lem-c-tm-for-3-conn}, $H_T \leq_{c} G \iff H_T^* \leq_{tm} G^*$. Therefore, $H_T$ is a contraction of a triangulated grid.

For the backward implication, if $G$ is a contraction of a triangulated graph, then $G \in {\cal C}$ by the definition of $\cal C$.
\end{proof}

\subsection{Minimal forbidden contraction}

In this subsection, we characterize our class ${\cal C}$ in terms of minimal forbidden contractions. Such a characterization already exists for the class of planar
graphs as shown by Demaine, Hajiaghayi, and Kawarabayashi~\cite{DemaineHK09}.
Before we present their characterization (we need it for ours), we first 
define some graph terminology.
Let $\{a_1,a_2,a_3\}$ and $\{b_1,b_2,b_3\}$ denote the partition classes of 
$K_{3,3}$. Adding edge $\{a_1,a_2\}$ leads to the graph $K_{3,3}^1$,
adding edges $\{a_1,a_2\},\{a_2,a_3\}$ to the graph $K_{3,3}^2$ and 
adding edges $\{a_1,a_2\},\{a_2,a_3\}, \{a_1,a_3\}$ to the graph $K_{3,3}^3$.
The 5-vertex {\it wheel} $W_4$ is obtained from adding an extra 
vertex adjacent to every vertex of a $C_4$.

\begin{figure}
  \centering
     \scalebox{1}{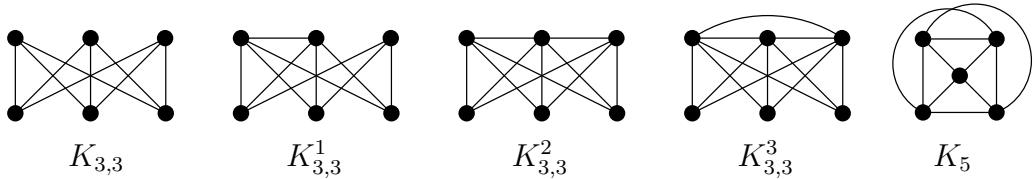}
  \caption{The minimal forbidden contractions in Theorem~\ref{t-planar1}.}\label{f-planar1}
\end{figure}

 \begin{figure}
  \centering
    \scalebox{1}{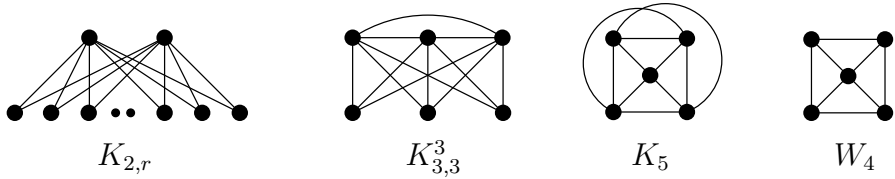}
  \caption{The minimal forbidden contractions in Theorem~\ref{t-planar2}.}\label{f-planar2}
\end{figure}

We are now ready to present the two characterizations; also see Figures~\ref{f-planar1} and~\ref{f-planar2}.

\begin{theorem}[Corollary 29 in \cite{DemaineHK09}]\label{t-planar1}
A connected graph is planar if and only if it does not contain any
graph from $\{K_{3,3},K_{3,3}^1,K_{3,3}^2,K_{3,3}^3,K_5\}$ as a contraction.
\end{theorem}

\begin{theorem}\label{t-planar2}
A graph belongs to ${\cal C}$ if and only if it does not contain any graph from $\{K_{2,r}\;|\; r\geq 2\}\cup \{K_{3,3}^3,K_5,W_4\}$ as a contraction. 
\end{theorem}

\begin{proof}
We first prove the ``$\Rightarrow$'' implication.
Suppose $H$ is a graph in ${\cal C}$. 
Because ${\cal C}$ is contraction-closed and every graph in
$\{K_{2,r}\;|\; r\geq 2\}\cup \{K_{3,3}^3,K_5,W_4\}$ does not belong
to ${\cal C}$, we find that $H$ cannot be contracted to such a graph.

We now prove the ``$\Leftarrow$'' implication. Let $H$ be a graph. Suppose 
$H$ does not contain any graph from $\{K_{2,r}\;|\; r\geq 2\}\cup \{K_{3,3}^3,K_5,W_4\}$ as a contraction. We must show that $H\in {\cal C}$. 
In order to derive a contradiction, suppose $H\notin {\cal C}$. We may
without loss of generality assume that $H$ is {\it minimal}, i.e., contracting an arbitrary edge in $H$ results in a graph $H'\in {\cal C}$.

\medskip
\noindent
{\it Claim 1. $H$ is planar.}

\medskip
\noindent
We prove Claim 1 as follows. Suppose $H$ is not planar. Then $H$ can
be contracted to a graph $F\in \{K_{3,3},K_{3,3}^1,K_{3,3}^2,K_{3,3}^3,K_5\}$ due to Theorem~\ref{t-planar1}.
By our assumptions, $F\notin \{K_{3,3}^3,K_5\}$. Hence, $F\in
\{K_{3,3},K_{3,3}^1,K_{3,3}^2\}$. However, in all these three cases, 
$F$ contains $W_4$ as a contraction by contracting the edge $\{a_2,b_2\}$.
This is not possible. Hence, we have proven Claim 1.

\medskip
\noindent
{\it Claim 2. $H$ is 3-connected.}

\medskip
\noindent
We prove Claim 2 as follows. Suppose $H$ is not 3-connected. Then $H$ either
contains a cut vertex $x$ or a cut set $\{x,y\}$ of size two. 

Consider the first case. Because $H\notin {\cal C}$,
there exists a component $D$ of $H-x$ such that $D+x\notin {\cal C}$; otherwise $H\in {\cal C}$ because of
Lemma~\ref{l-pasting}. We observe that $D+x$ contains no graph from
$\{K_{2,r}\;|\; r\geq 2\}\cup \{K_{3,3}^3,K_5,W_4\}$ as a contraction, 
because otherwise $H$ would contain such a graph as contraction as well, and this is not possible. Hence, we could take $D+x$ instead of $H$. This means that $H$ is not minimal, a contradiction.
 
Now consider the second case. We find that $H$ is not minimal either by the same kind of arguments as in the first case, unless $x$ and $y$ are not adjacent. However, in that case contracting the connected components of $H-\{x,y\}$ to a single vertex yields the graph $K_{2,r}$, where $r$ denotes the number of such components. This is not possible, and Claim 2 has been proven.

\medskip
\noindent
{\it Claim 3. $H$ allows an embedding with an outer face of exactly 4 vertices.}

\medskip
\noindent
We prove Claim 3 as follows. Because $H$ is 3-connected (by Claim 2) and $H\notin {\cal C}$, we find that $H$ is not triangulated. This means that $H$ contains a face
of at least 4 vertices. Without loss of generality we may assume that this face is the outer face. If it contains more than 4 vertices, we can contract one of its edges and obtain a graph $H'$ that contradicts the minimality of $H$.
This proves Claim 3.

\medskip
Now let $C$ be the outer face of $H$. By Claim 3, we may assume that $C$ consists of exactly 4 vertices $z_1,\ldots,z_4$.
Because $H$ is planar by Claim 1, $H-C$ contains at most one 
connected component adjacent to all four vertices of $C$. If this is the case we 
can contract this component to one single vertex $c$. We contract the other connected components of $H-C$ to single vertices as well. This leads to a graph in which the vertices $c,z_1,\ldots,z_4$ induce a $W_4$. We get rid of any 
remaining vertex $c'$ as follows. If $c'$ is adjacent to
$c$ then we contract the edge $\{c,c'\}$. Otherwise, $c'$ is adjacent to 
a vertex $z_i$ or two vertices $z_i,z_j$ of $C$, and in the latter case $z_i$ and $z_j$ are adjacent. We contract the edge
$\{c',z_i\}$. In this way we find that $H$ can be contracted to $W_4$. This is not possible.

If $H-C$ contains no component adjacent to all four vertices of $C$, then $H$
contains $C_4$ as a contraction by similar arguments as used above.
This completes the proof of Theorem~\ref{t-planar2}.
\end{proof}

We observe that none of the graphs $\{K_{2,r}\;|\; r\geq 2\}\cup \{K_{3,3}^3,K_5,W_4\}$ is a contraction of another. Thus, Theorem \ref{t-planar2} characterizes $\cal C$ in terms of minimal forbidden contractions.

\section{Conclusions}

\begin{enumerate}

\item This paper can be read as an introductory study of the class  $\cal C$. We define the class, discover some of its properties and provide an algorithmic application. The graphs that belong to $C$ are exactly those that are contractions of triangulated graphs, or equivalently, contractions of triangulated grids. We believe that $\cal C$ is an interesting class of graphs in its own right.

\item Note that membership in $\cal C$ can be tested in polynomial time. The input graph should be decomposed along vertex- and $K_2$-cuts and each component tested for being a planar triangulated graph. 

\item Lemma \ref{lem-traingulated-grid} tells us that a graph belongs to $\cal C$ if and only if it is a contraction of a triangulated grid. This could be viewed as an analog to the well-known statement that a graph is planar if and only if it is a minor of a grid. 

\item Theorem \ref{thm-big-twd-triang} tells us that a graph $H \in \cal C$ if and only if there exists a constant $c_H$ such that if the tree-width of a graph is at least $c_H$, it contains $H$ as a contraction.





\item The well-known theorem by Kruskal states that the set of trees over a well-quasi-ordered set of labels is itself well-quasi-ordered \cite{Kruskal60}. The set of triangulated planar graphs is known to be well-quasi ordered \cite{DemaineHK09}. Every graph in $\cal C$ can be seen as a tree (a decomposition tree with respect to vertex- and $K_2$-cuts) whose vertices are labelled by triangulated graphs. Hence, the graphs in $\cal C$ are well-quasi-ordered with respect to contractions. 

\end{enumerate}

{
\bibliographystyle{plain}
\bibliography{Contractions-in-planar-graphs-FPT}
}

\end{document}